\documentclass[12pt]{article}
\usepackage{amsmath,amssymb,color,amscd}
\usepackage{xspace}

\newcommand{\Map}{{\rm{Map}_{\mathbb{C}}}}

\newcommand{\Var}{{\rm{Var}_{\mathbb{C}}}}

\def\Ex{\mathbf{Exp}}
\def\Lo{\mathbf{Log}}

\def\uu{{\underline{u}}}

\def\kk{{\underline{k}}}

\def\1{\underline{1}}

\def\AA{{\mathbb A}}

\def\LLL{{\mathbb L}}
\def\Z{{\mathbb Z}}

\def\C{{\mathbb C}}

\def\Ex{{\bf{Exp}}}
\def\Lo{{\bf{Log}}}

\newtheorem{theorem}{Theorem}

\newtheorem{proposition}{Proposition}

\newenvironment{definition}
{\smallskip\noindent{\bf Definition\/}:}{\smallskip\par}
\newenvironment{corollary}
{\smallskip\noindent{\bf Corollary\/}.}{\smallskip\par}

\newenvironment{remark}
{\smallskip\noindent{\bf Remark\/}.}{\smallskip\par}

\newenvironment{proof}
{\noindent{\bf Proof\/}.}{{ $\square$}\smallskip\par}

\title{Power structure over the Grothendieck ring of maps
\footnote{Math. Subject Class.: 14A10, 18F30, 55M35. Keywords: 
lambda-structure, power structure, complex quasi-projective varieties,
maps, Grothendieck ring.}
}

\author{S.M.~Gusein-Zade \thanks{The work of the first author
(Sections~\ref{sec:Grothendieck-of-maps}, \ref{sec:effective} and~\ref{sec:versions}) 
was supported
by the grant 16-11-10018 of the Russian Science Foundation.
Address: Moscow State University, Faculty
of Mathematics and Mechanics, GSP-1, Moscow, 119991, Russia. E-mail:
sabir\symbol{'100}mccme.ru} \and I.~Luengo \thanks{The last two authors were partially
supported by a competitive Spanish national grant MTM2016-76868-C2-1-P.
Address:  ICMAT (CSIC-UAM-UC3M-UCM), Dept. of Algebra, Complutense University of Madrid,
Plaza de Ciencias 3, Madrid, 28040, Spain.
E-mail: iluengo\symbol{'100}mat.ucm.es} \and
A.~Melle-Hern\'andez \thanks{Address:  Instituto de Matem\'a¡tica Interdisciplinar (IMI), Dept. of Algebra, Complutense University of Madrid,
Plaza de Ciencias 3, Madrid, 28040, Spain. E-mail: amelle\symbol{'100}mat.ucm.es}}
\date{}
\begin{document}
\def\eps{\varepsilon}

\maketitle

\begin{abstract}
A power structure over a ring is a method to give sense to
expressions of the form $(1+a_1t+a_2t^2+\ldots)^m$, where $a_i$, $i=1, 2,\ldots$,
and $m$ are elements of the ring. The (natural) power structure over the Grothendieck ring
of complex quasi-projective varieties appeared to be useful for a number of applications.
We discuss new examples of $\lambda$- and power structures over some Grothendieck
rings of varieties. The main example is for  the Grothendieck ring of maps of
complex quasi-projective varieties. We describe two natural $\lambda$-structures on it
which lead to the same power structure. 
We show that this power structure is  effective. 
In the terms of this power structure we write some equations containing classes of Hilbert--Chow morphisms.
We describe some generalizations of this construction for maps of varieties with some additional structures.
\end{abstract}

\section{Introduction}\label{sec:Intro}
A $\lambda$-{\em structure} (called sometimes a pre-$\lambda$-structure) on a ring $R$ 
is an additive-to-multiplicative homomorphism $R\to 1+tR[[t]]$ (\cite{Knutson}).
A {\em power structure} over a ring $R$ (\cite{GLM-MRL}) is a method to give sense to
expressions of the form $(1+a_1t+a_2t^2+\ldots)^m$, where $a_i$, $i=1, 2,\ldots$,
and $m$ are elements of the ring $R$ (also as a series from $1+tR[[t]]$).
The notions of $\lambda$-structures and
power structures are closely related to each other, but are not equivalent.
In particular, each $\lambda$-structure on a ring defines a power structure over it,
but there are, in general, many $\lambda$-structures corresponding to one and the same
power structure. A ``natural'' power structure over the Grothendieck ring $K_0(\Var)$ of
complex quasi-projective varieties (see, e.g., \cite{Craw}) was described in \cite{GLM-MRL}. Its version for
the relative case (that is, over the Grothendieck ring of complex quasi-projective varieties
over a fixed variety) was defined in \cite{GLMSteklov}. A power structure over the
Grothendieck ring of complex quasi-projective varieties over an Abelian monoid was
defined in \cite{M_Sh}. (A particular case of the Grothendieck ring of varieties over an Abelian
monoid when the monoid is the Abelian group $\C$ was considered in \cite{Loeser1, Loeser2, Wyss}
under the name Grothendieck rings of varieties with exponentials.)
Power structures over Grothendieck rings of varieties appear
to be useful, in particular, for formulation and proof of formulae for the generating
series of classes of some configuration spaces or of their invariants, see, e.g.,
\cite{Michigan, BBS, BrMor, M_Sh}.

An important property of the power structure over the Grothendieck ring $K_0(\Var)$
which makes it useful for the mentioned applications is its
{\em effectiveness}. This means that if all the coefficients $a_i$ of the series
$A(t)=1+a_1t+a_2t^2+\ldots$ and the exponent $m$ are classes of complex
quasi-projective varieties (not of virtual ones: differences of such classes),
then all the coefficients of the series $(1+a_1t+a_2t^2+\ldots)^m$ are also
represented by classes of complex quasi-projective varieties. This is a somewhat
special property of this power structure. Another natural power structure over
the Grothendieck ring $K_0(\Var)$ (in fact up to now only two power structures over
this ring are known) and also its natural ``extension'' to the Grothendieck ring of
stacks (\cite{Stacks}) are not effective.

Here we discuss new examples of $\lambda$- and power structures over some Grothendieck
rings of varieties. 
The main example is for  the Grothendieck ring of maps of
complex quasi-projective varieties. We describe two natural $\lambda$-structures on it
which lead to the same power structure. 
We show that this power structure is  effective. 
In the terms of this power structure we write some equations containing classes of Hilbert--Chow morphisms.
We describe some generalizations of this construction for maps of varieties with some additional structures.

We are  thankful to a referee of the initial version of this paper submitted to another journal
for finding a crucial mistake in it.


\section{Power structures and $\lambda$-structures}\label{sec:Power_lambda}
A {\em power structure} over a ring $R$ is a method to give sense to
expressions of the form $\left(A(t)\right)^m$, where
$A(t)=1+a_1t+a_2t^2+\ldots$ is a formal series with the coefficients $a_i$ from $R$,
the exponent $m$ is also an element of $R$.

\begin{definition} (\cite{GLM-MRL})
 A {\em power structure} over a ring $R$ with unity $1$ is a map
 $$
 \left(1+tR[[t]]\right)\times R\to 1+tR[[t]]
 $$
 $\left(\left(A(t), m\right)\mapsto\left(A(t)\right)^m\right)$ which possesses the 
 properties of the exponential function, namely:
 \begin{enumerate}
\item[(1)]  $\left(A(t)\right)^0=1$,
\item[(2)]  $\left(A(t)\right)^1=A(t)$,
\item[(3)]  $\left(A(t)\cdot B(t)\right)^{m}=\left(A(t)\right)^{m}\cdot
\left(B(t)\right)^{m}$,
\item[(4)]  $\left(A(t)\right)^{m+n}=\left(A(t)\right)^{m}\cdot
\left(A(t)\right)^{n}$,
\item[(5)]  $\left(A(t)\right)^{mn}=\left(\left(A(t)\right)^{n}\right)^{m}$,
\item[(6)] $(1+t)^m=1+mt+$ terms of higher degree,
\item[(7)]  $\left(A(t^k)\right)^m =
\left(A(t)\right)^m\raisebox{-0.5ex}{$\vert$}{}_{t\mapsto t^k}$.
\end{enumerate}
\end{definition}

\begin{remark} 
In \cite{GLM-MRL} the properties~$(6)$ and $(7)$ were not required, though the
constructed power structures possessed them.
\end{remark}

\begin{definition}
 A power structure over a ring $R$ is {\em finitely determined} if the fact that
 two series $A_1(t)$ and $A_2(t)$ from $1+tR[[t]]$ differ by terms of degree $\ge k$
 (that is $A_1(t)-A_2(t)\in \mathfrak{m}^k$, where $\mathfrak{m}=\langle t\rangle \subset R[[t]]$)
 implies that $\left(A_1(t)\right)^m-\left(A_2(t)\right)^m\in \mathfrak{m}^k$.
\end{definition}

A natural power structure over the ring $\Z$ of integers (and the only one) is
defined by the usual equation for an exponent of a series (see, e.g., \cite{St}, page~40)
\begin{eqnarray}
  \hspace{-20pt}&\ & (1+a_1t+a_2t^2+\ldots)^m=\nonumber\\
  \hspace{-20pt}&=&1+\sum_{k=1}^{\infty}\left(\sum_{\{k_i\}:\sum_i ik_i=k}
  \frac{m(m-1)\ldots(m-\sum_i k_i+1)\times\prod_i a_i^{k_i}}{\prod_i k_i!}  
  \right)\cdot t^k. \label{power-Z}
\end{eqnarray}
(Obviously this power structure is finitely determined.)

The power structure over the Grothendieck ring $K_0(\Var)$ of complex quasi-projective
varieties defined in \cite{GLM-MRL} is given by the equation
  \begin{eqnarray}\label{power-Grothendieck}
  \hspace{-20pt}&\ & (1+[A_1]t+[A_2]t^2+\ldots)^{[M]}=\nonumber\\
  \hspace{-20pt}&=&1+\sum_{k=1}^{\infty}\left(
  \sum_{\{k_i\}:\sum_i ik_i=k}
  \left[
  \left.
  \left(
  \left(
  M^{\sum_i k_i}\setminus \Delta
  \right)
  \times\prod_i A_i^{k_i}
  \right)
  \right/
  \prod_i S_{k_i}
  \right]
  \right)
  \cdot t^k\,,\label{power-K_0}
 \end{eqnarray}
where $[A_i]$, $i=1, 2, \ldots$, and $[M]$ are classes in $K_0(\Var)$ of complex
quasi-projective varieties, $\Delta$ is the ``large diagonal'' in $M^{\sum_i k_i}$,
that is the set of (ordered) collections of $\sum_i k_i$ points from $M$ with at least
two coinciding ones, the group $S_{k_i}$ of permutations on $k_i$ elements acts by
simultaneous permutations on the components of the corresponding factor $M^{k_i}$ in
$M^{\sum_i k_i}=\prod_i M^{k_i}$ and on the components of the factor $A_i^{k_i}$.
 
Except the Grothendieck ring of complex quasi-projective varieties one can consider the
Grothendieck semiring $S_0(\Var)$. It is defined in the same way as $K_0(\Var)$ with the word group
substituted by the word semigroup. Elements of the semiring $S_0(\Var)$ are represented by
``genuine'' complex quasi-projective varieties, not by virtual ones (that is formal differences
of varieties). Two complex quasi-projective varieties $X$ and $Y$ represent the same element
of the semiring $S_0(\Var)$ if and only if they are piece-wise isomorphic, that is if there
exist decompositions $X =\bigsqcup_{i=1}^sX_i$ and $Y =\bigsqcup_{i=1}^sY_i$ into Zariski locally
closed subsets such that $X_i$ and $Y_i$ are isomorphic for $i = 1, \ldots , s$. There is a natural map
(a semiring homomorphism) from $S_0(\Var)$ to $K_0(\Var)$. (It is not known whether or not this
map is injective.)

A power structure over the Grothendieck ring $K_0(\Var)$ is called {\em effective} if the
fact that all the coefficients $a_i$ of the series $A(t)$ and the exponent $m$ are represented
by classes of complex quasi-projective varieties (i.e., belong to the image of the map
$S_0(\Var)\to K_0(\Var)$) implies that all the coefficients of the series $\left( A(t)\right)^m$
are also represented by such classes. Roughly speaking this means that the power structure can be
defined over the Grothendieck semiring $S_0(\Var)$. The same concept is used for Grothendieck
rings of complex quasi-projective varieties with additional structures. The effectiveness of the
described power structure over the Grothendieck ring $K_0(\Var)$ is clear from Equation~(\ref{power-K_0}).

An equation similar to (\ref{power-K_0}) was given in \cite{M_Sh} for a power structure over the
Grothendieck ring of complex quasi-projective varieties over an Abelian monoid used there.

Power structures over a ring are related to $\lambda$-structures on it.
Let $R$ be a ring with a $\lambda$-structure, that is, for each $a\in R$ there is defined a series
$\lambda_a(t)=1+at+\ldots\in 1+tR[[t]]$ so that $\lambda_{a+b}(t)=\lambda_a(t)\lambda_b(t)$
(in other words one has an additive-to-multiplicative homomorphism $R\to 1+tR[[t]]$;
see, e.g., \cite{Knutson}). A $\lambda$-structure $\lambda_a(t)$ defines a
(finitely determined) power structure
over $R$ in the following way. Any power series $A(t)\in 1+tR[[t]]$ can be in a unique way represented as
the product $A(t)=\prod\limits_{i=1}^{\infty}\lambda_{b_i}(t^i)$ with $b_i\in R$.
Then one defines the series $\left(A(t)\right)^m$ by
\begin{equation}\label{lambda-power}
 \left(A(t)\right)^m:=\prod\limits_{i=1}^{\infty}\lambda_{mb_i}(t^i)\,.
\end{equation}

\begin{remark}
 A $\lambda$-structure on a ring $R$ defines maps ${\Ex}: tR[[t]]\to 1+tR[[t]]$ and
 ${\Lo}: 1+tR[[t]]\to tR[[t]]$ (inverse to each other) in the following way:
 $$
 {\Ex}(b_1t+b_2t^2+\ldots):=\prod_{k\ge 1}\lambda_{b_k}(t^k)\,;
 $$
 if $1+a_1t+a_2t^2+\ldots=\prod\limits_{k=1}^{\infty}\lambda_{b_k}(t^k)$, then
 $$
 {\Lo}(1+a_1t+a_2t^2+\ldots):=\sum\limits_{k=1}^{\infty}b_kt^k\,.
 $$
 The map ${\Ex}$ is an additive-to-multiplicative homomorphism.
 The map ${\Lo}$ is a multiplicative-to-additive homomorphism. 
 Each of these maps determines the $\lambda$-structure on the ring, see, e.g., \cite{GLM-MRL}. 
\end{remark}

One can show that the power structure (\ref{power-K_0}) over the Grothendieck
ring $K_0(\Var)$ corresponds to the $\lambda$-structure on it defined by the
\emph{Kapranov motivic  zeta function} (\cite{K})
$$
\zeta_{[X]}(t)=1+[X]t+[S^2X]t^2+[S^3X]t^3+\ldots,
$$
where $S^kX=X^k/S_k$ is the $k$th symmetric power of the variety $X$.
In terms of the power structure one has $\zeta_{[X]}(t)=(1+t+t^2+\ldots)^{[X]}=(1-t)^{-[X]}$.

There are many $\lambda$-structures corresponding to one and the same power structure
over a ring $R$. For any series $\lambda_1(t)=1+t+a_2t^2+\ldots$ the equation
$$
\lambda_a(t):=\left(\lambda_1(t)\right)^a
$$
gives a $\lambda$-structure on the ring $R$. For example, the power structure (\ref{power-K_0}) over $K_0(\Var)$ 
can be defined both by the   Kapranov zeta function and by the
$\lambda$-structure $$\lambda_{[X]}(t):=(1+t)^{[X]}=1+[X]t+[B_2 X]t^2+[B_3X]t^3+\ldots,$$
where $B_k X:=(X^k\setminus \Delta)/S_k$ is the configuration space of $k$ distinct
unordered points of $X$ (see~\cite{GLM-MRL}).

Another ``natural'' $\lambda$-structure on the Grothendieck ring $K_0(\Var)$ 
({\em opposite} to the one defined by the Kapranov zeta function $\zeta_{[X]}(t)$) is defined
by the series $\zeta_{-[X]}(-t)$. One can show that the corresponding power structure
over the ring $K_0(\Var)$ is not effective (see \cite{Stacks}). (It seems that one knows
no power structures over the ring $K_0(\Var)$ except the described two.)

Let $R_1$ and $R_2$ be rings with power structures over them. A ring
homomorphism $\varphi:R_1\to R_2$ induces the natural homomorphism
$R_1[[t]]\to R_2[[t]]$ (also denoted by $\varphi$) by
$\varphi\left(\sum_i a_it^i\right)=\sum_i\varphi(a_i)\,t^i$.
One has the following statement.

\begin{proposition}\cite{Michigan}\label{prop2}
If a ring homomorphism $\varphi:R_1\to R_2$ is such that
$\varphi\left((1-t)^{-a}\right)=(1-t)^{-\varphi(a)}$, then
$\varphi\left(\left(A(t)\right)^m\right)=
\left(\varphi\left(A(t)\right)\right)^{\varphi(m)}$.
\end{proposition}

Equations written in terms of the power structure (\ref{power-K_0}) give equations
for the Euler characteristics with compact support $\chi(\bullet)$ and for the \emph{Hodge--Deligne polynomial}
$e_{\bullet}(u,v)$ via the natural homomorphisms $\chi:K_0(\Var)\to\Z$ and
$e:K_0(\Var)\to\Z[u,v]$. These homomorphisms are compatible with the power
structures over the rings $\Z$ (see Equation~(\ref{power-Z})) and $\Z[u,v]$, where the
power structure over the latter is defined as follows. 

Let $\Z[u_1,\ldots, u_r]$ be the ring of polynomials in $r$ variables.
Let $P(u_1,\ldots,u_r)=
\sum\limits_{\kk\in\Z_{\ge0}^r}p_\kk \uu^\kk\in \Z[u_1,\ldots,u_r]$,
where $\kk=(k_1,\ldots,k_r)$,
$\uu=( u_1,\ldots,u_r)$, $\uu^\kk =u_1^{k_1}\cdot\ldots\cdot u_r^{k_r}$,
$ p_\kk\in\Z$. Let
$$\lambda_P(t):=
\prod\limits_{\kk\in\Z_{\ge0}^r} (1-\uu^\kk t)^{-p_\kk},$$
where the power (with an integer exponent $-p_\kk$) means the usual one.
The series $\lambda_P(t)$ defines a $\lambda$-structure on the ring $\Z[u_1,\ldots, u_r]$
and therefore a power structure over it (with $\lambda_P(t)=(1-t)^{-P}$).

Let $r=2$, $u_1=u$, $u_2=v$. Let $e: K_0({\Var})\to {\Z}[u,v]$
be the ring homomorphism which sends the
class $[X]$ of a quasi-projective variety $X$ to its Hodge--Deligne
polynomial $e_X(u,v)=\sum_{i,j} h_X^{ij}(-u)^i(-v)^j$. 
One can see that the homomorphism $e$ respects the $\lambda$- and therefore the
power structures over the source and over the target. This is shown in \cite{Cheah},\cite{Michigan}:
in terms of the power structures Proposition 1.2  in \cite{Cheah} can be rewritten as
$$
e\left((1-t)^{-[X]}\right)=(1-t)^{-e_X(u,v)}.
$$

\section{Grothendieck ring of maps}\label{sec:Grothendieck-of-maps}

Let us consider (regular) maps $f:X \to Y$ between complex quasi-projective
varieties.  

\begin{definition}
 Maps $f:X \to Y$ and  $f':X' \to Y'$ between complex quasi-projective varieties are \emph{equivalent}
 if there exist isomorphisms $h_1:X\to X'$ and $h_2:Y\to Y'$ such that $h_2 \circ f=f'\circ h_1$.
\end{definition}

The definition of the Grothendieck
ring of maps is inspired by the one for the Grothendieck ring of varieties $K_0(\Var)$.

\begin{definition} The Grothendieck ring $K_0(\Map)$ of maps between complex quasi-projective
varieties is the free Abelian group generated by the classes $[X \stackrel{f}{\rightarrow} Y]$ of maps modulo the relations:
 \begin{enumerate}
 \item[1)] if two maps $f:X \to Y$ and  $f':X' \to Y'$ are  equivalent,
 then 
 $$
 [X \stackrel{f}{\rightarrow} Y]=[X' \stackrel{f'}{\rightarrow} Y'];
 $$
 \item[2)] if $f:X \to Y$ is a map and $Z$ is a Zariski closed subset of $Y$, then 
$$
[X \stackrel{f}{\rightarrow} Y]=[f^{-1}(Z) \stackrel{\; \; f|_{f^{-1}(Z)} \; \,}{\longrightarrow} Z]
+[f^{-1}(Y\setminus Z) \stackrel{\; \; f|_{f^{-1}(Y\setminus Z)} \; \,}{\longrightarrow} Y\setminus Z];
$$
 \item[3)]  if $f:X \to Y$ is a map and $Z$ is a Zariski closed subset of $X$, then 
$$
[X \stackrel{f}{\rightarrow} Y]=[Z \stackrel{\; \; f|_{Z} \; \,}{\longrightarrow} Y]
+[X\setminus Z \stackrel{\; \; f|_{X\setminus Z} \; \,}{\longrightarrow} Y].
$$
\end{enumerate}
\end{definition}

The definition  means that the summation in $K_0(\Map)$ can be defined by the disjoint union, that is  
$$
[X_1 \stackrel{f_1}{\rightarrow} Y_2] +[X_2 \stackrel{f_2}{\rightarrow} Y_2]:=
[X_1 \sqcup  X_2 \stackrel{f_1 \sqcup f_2}{\longrightarrow} Y_1 \sqcup Y_2].
$$
The multiplication in $K_0(\Map)$ is defined by the Cartesian product:
$$[X_1 \stackrel{f_1}{\rightarrow} Y_2] \cdot [X_2 \stackrel{f_2}{\rightarrow} Y_2]:=
[X_1 \times  X_2 \stackrel{f_1 \times f_2}{\longrightarrow} Y_1 \times  Y_2].$$
The unit $\boldsymbol{1}$ in $K_0(\Map)$ is represented by the identity map form a 
point to itself.

\begin{remark}
1) The relation $3)$ applied to $Z=\emptyset$ implies that $[\emptyset \to Y]=0$.

2) Since, for a scheme $X$, its reduction  $X_{red}$ is a Zariski closed subset of $X$, then using 2) and 3) one has
 $[X \stackrel{id}{\rightarrow} X]=[X_{red} \stackrel{id}{\rightarrow} X_{red}]$.
\end{remark}

There is a natural homomorphism from $K_0(\Map)$ to the Grothendieck ring $K_0(\Var)$ of complex
quasi-projective varieties 
which sends $[X \stackrel{f}{\rightarrow}Y]$  to $[X]$.
(The correspondence $[X \stackrel{f}{\rightarrow}Y]\mapsto [Y]$ is not well defined.) 
This map has two natural ``sections'':
injective  ring homomorphisms
$K_0(\Var)\to K_0(\Map)$ defined by $[X] \mapsto [X \stackrel{id}{\rightarrow} X] $ and
$[X]\mapsto [X {\rightarrow}\ pt ]$ respectively, where $pt$ is a one point set.

In the same way (substituting the word group by the word semigroup) one can define the Grothendieck semiring 
$S_0(\Map)$ of maps between complex quasi-projective varieties. The elements in $S_0(\Map)$ are represented
by classes of genuine maps, not by virtual ones (that is formal differences of maps). The classes in $S_0(\Map)$ of
regular  maps $f_1:X_1 \to Y_1$ and $f_2:X_2 \to Y_2$ are equal if and only if they are piece-wise isomorphic,
that is  if there exist partitions $X_1= \bigsqcup_{i=1}^n X_{1,i}$ and 
$X_2= \bigsqcup_{i=1}^n X_{2,i}$ such that the maps 
$f_1:X_{1,i} \to f_1(X_{1,i})$ and 
$f_2:X_{2,i} \to  f_2(X_{2,i})$ are  equivalent for all $i=1, \ldots, n$.

\begin{remark}
It is easy to see that the subring  of the Grothendieck ring of $K_0(\Map)$ generated by the classes of maps
of zero-dimensional varieties (that is of finite sets) is isomorphic to $\Z$. 
\end{remark}

\section{$\lambda$-structures over the Grothendieck ring of maps}\label{sec:lambdas-on-maps}

Let us describe two natural $\lambda$-structures over the ring $K_0(\Map)$.

For a map $f:X\to Y$, one has the natural map $S^k f:S^k X \to S^k Y$
between the $k$th symmetric powers of $X$ and $Y$.  
(Pay attention that the map $f$ does not define a map between the configuration spaces
$B_k X=\left(X^k \setminus \Delta\right)/S_k$ and $B_k Y=\left(Y^k \setminus \Delta\right)/S_k$ of
$k$ distinct points on $X$ and $Y$ respectively.)

\begin{definition}
 The \emph{Kapranov motivic zeta function of a map} $f:X\to Y$ is defined by
\begin{equation}\label{Kapranov-maps}
 \zeta_{[X \stackrel{f}{\rightarrow} Y]}(t):=\boldsymbol{1}+\sum_{k\geq 1} 
[S^k X \stackrel{S^k f}{\longrightarrow} S^k Y]\cdot t^k \in \boldsymbol{1}+t K_0(\Map)[[t]].
\end{equation}
\end{definition}

\begin{proposition}\label{Kapranov-lambda}
The Kapranov motivic zeta function defines a $\lambda$-structure on the ring $K_0(\Map)$.
 \end{proposition}

\begin{proof}
It is necessary to show that, for two maps $f_1:X_1\to Y_1$ and $f_2:X_2\to Y_2$, one has 
$$
\zeta_{[X_1 \sqcup X_2 \stackrel{f_1 \sqcup f_2}{\longrightarrow} Y_1 \sqcup Y_2]}(t)
=\zeta_{[X_1 \stackrel{f_1}{\rightarrow} Y_1]}(t) \cdot \zeta_{[X_2 \stackrel{f_2}{\rightarrow} Y_2]}(t).
$$ 
This follows from the following obvious relation
$$
S^k(X_1 \sqcup X_2) \stackrel{S^k (f_1 \sqcup f_2)}{\longrightarrow} S^k (Y_1 \sqcup Y_2)=
\bigsqcup_{i=0}^k \left( S^i X_1 \stackrel{S^i (f_1 )}{\longrightarrow} S^i Y_1\right) \times 
\left(S^{k-i} X_2 \stackrel{S^{k-i} (f_2 )}{\longrightarrow} S^{k-i} Y_2\right).
$$
\end{proof}

Let $B_kX:=(X^k\setminus\Delta)/S_k$ be the configuration space of collections of $k$  different
points in $X$ ($\Delta$ is the big diagonal in $X_k$ consisting of $k$-tuples of points of $X$ with
at least two coinciding ones).
For a map $f:X\to Y$, one has the corresponding map $B_{k}f:B_{k}X\to S^kY$
from the configuration space of $k$ distinct points on $X$ to the $k$th symmetric power of the variety $Y$.
Let
\begin{equation}\label{lambda2-maps}
\lambda_{[X \stackrel{f}{\rightarrow} Y]}(t):=\boldsymbol{1}+\sum_{k\geq 1} 
[B_{k}X \stackrel{B_{k}f}{\longrightarrow}  S^kY]\cdot t^k \in \boldsymbol{1}+t K_0(\Map)[[t]].
\end{equation}

\begin{proposition}\label{lambda-lambda2}
The series $\lambda_{[X \stackrel{f}{\rightarrow} Y]}(t)$ defines a $\lambda$-structure on
the ring $K_0(\Map)$.
\end{proposition}

\begin{proof}
Just as in Propositions~\ref{Kapranov-lambda}, 
for two maps $f_1:X_1\to Y_1$ and $f_2:X_2\to Y_2$, one has 
$$
B_{k}\left(X_1 \sqcup X_2\right) \stackrel{B_{k}(f_1 \sqcup f_2)}\longrightarrow
S^k\left(Y_1 \sqcup Y_2\right)
=\bigsqcup_{i=0}^k (B_{i}X_1 \stackrel{B_{i}f_1}\longrightarrow S^iY_1)\times
(B_{k-i}X_2 \stackrel{B_{k-i}f_2}\longrightarrow S^{k-i}Y_2).
$$
\end{proof}

\section{A  power structure over the ring $K_0(\Map)$}\label{sec:effective}

\begin{theorem}\label{lambda-effective}
The $\lambda$-structures given  by the series 
$\zeta_{[X \stackrel{f}{\rightarrow} Y]}(t)$ and $\lambda_{[X \stackrel{f}{\rightarrow} Y]}(t)$
define one and the same power structure over the ring $K_0(\Map)$. 
Moreover, this power structure is effective.
\end{theorem}

\begin{proof}
To prove the statement we shall give a direct (geometric) description of the corresponding power structure.
For a series 
$$
A(t):=\boldsymbol{1}+[X_1 \stackrel{f_1}{\rightarrow} Y_1] t+[X_2 \stackrel{f_2}{\rightarrow} Y_2] t^2+
\ldots \in \boldsymbol{1}+t K_0(\Map)[[t]],
$$
and for an element $m=[M \stackrel{f}{\rightarrow} N]\in K_0(\Map)$, let us define $(A(t))^m$ as
\begin{equation}\label{geometric-lambda}
 \boldsymbol{1}+
\sum_{k=1}^\infty
\left(
\sum_{\kk:\sum_i ik_i=k}
\left[\left(\left(M^{\sum_i k_i}\setminus\Delta \stackrel{f^{\sum_i k_i}}{\rightarrow} N^{\sum_i k_i}
\right)\times  
\prod_i (X_i \stackrel{f_i}{\rightarrow} Y_i)^{k_i}\right)\left/\prod_iS_{k_i}\right.
\right]
\right)
\cdot t^k,
\end{equation}
where $\kk=\{k_i: i\in\Z_{>0}, k_i\in\Z_{\ge0}\}$, $\Delta$ is the
"large diagonal" in $M^{\sum_i k_i}$ which consists of $(\sum_i k_i)$-tuples
of points of $M$ with at least two coinciding ones, the permutation group
$S_{k_i}$ acts simultaneously on the components of the factors $M^{k_i}$ and $N^{k_i}$ 
in $M^{\sum_i k_i}\setminus\Delta$ and in $N^{\sum_i k_i}$
and on the components of $(X_i \stackrel{f_i}{\rightarrow} Y_i)^{k_i}.$

In Equation  (\ref{geometric-lambda}), by 
$$
\left(\left(M^{\sum_i k_i}\setminus\Delta \stackrel{f^{\sum_i k_i}}{\rightarrow} N^{\sum_i k_i}
\right)\times  
\prod_i (X_i \stackrel{f_i}{\rightarrow} Y_i)^{k_i}\right)\left/\prod_iS_{k_i}\right.,
$$
we mean the map 
$$
\left(\left(M^{\sum_i k_i}\setminus\Delta \right) \times \prod_i X_i^{k_i}\right) \left/\prod_iS_{k_i}\right.
\to
\left(N^{\sum_i k_i}
\times  
\prod_i  Y_i^{k_i}\right)\left/\prod_iS_{k_i}\right.
$$
induced by $f$ and $f_i$.
(Pay attention that the action of the group $\prod_i S_{k_i}$ on the source
$\left(M^{\sum_i k_i}\setminus\Delta \right) \times \prod_i X_i$ is free.)

The fact that Equation~(\ref{geometric-lambda}) defines a power structure over the Grothendieck ring
$K_0(\Map)$ can be deduced from an interpretation of the coefficients of $t^k$ in it similar to the one
in \cite{GLM-MRL}. Let $X_0=pt$, let $\Gamma$ be the disjoint union $\bigsqcup\limits_{i=0}^{\infty} X_i$
and let $I:\Gamma\to\Z$ be the tautological function on $\Gamma$ which sends the component $X_i$ to $i$.
A representative of the coefficient of $t^k$ in (\ref{geometric-lambda}) can be identified with the
configuration space of maps $\psi:M\to\Gamma$ such that $\sum\limits_{x\in M}I(\psi(x))=k$.
This means that there are finitely many $x\in M$ such that $\psi(x)\not\in X_0$. Taking into account
only the points $x\in M$ with $\psi(x)\not\in X_0$, it is possible to say that we consider collections
of (non-coinciding) particles on $M$ with positive integer charges and with the space of internal states
of a particle with charge $i$ parametrized by the variety $X_i$, $i=1, 2, \ldots$ The coefficient of $t^k$
in (\ref{geometric-lambda}) is represented by the configuration space of collections of particles with
the total charge $k$ (cf. \cite{Gorsky}). These data correspond to the source of the map. The target is
represented by a similar configuration space of particles on $N$ (images of the particles on $M$ under the
map $f$) whose locations are permitted to coincide and whose spaces of internal states are parametrized
by the varieties $Y_i$. The map from the source to the target is defined in the obvious way (and is
determined by the maps $f$ and $f_i$).

Let us prove the properties (3)--(5) from the definition of a power structure (cf.~\cite{GLM-MRL}).

(3) Let
$$
A(t):=\boldsymbol{1}+[X_1 \stackrel{f_1}{\rightarrow} Y_1] t+[X_2 \stackrel{f_2}{\rightarrow} Y_2] t^2+
\ldots,
$$
$$
B(t):=\boldsymbol{1}+[X'_1 \stackrel{f'_1}{\rightarrow} Y'_1] t+[X'_2 \stackrel{f'_2}{\rightarrow} Y'_2] t^2+
\ldots,
$$
and $m=[M \stackrel{f}{\rightarrow} N]$.
The coefficient of $t^s$ in $A(t)B(t)$ is equal to
$$
\sum_{i=0}^s[X_i\times X'_{s-i}\to Y_i\times Y'_{s-i}]\,.
$$
The coefficient of $t^k$ in $\left(A(t)B(t)\right)^m$ is represented by the configuration space
of maps $\psi$ from $M$ to
$$
\bigsqcup_{s=0}^{\infty}\bigsqcup_{i=0}^s (X_i\times X'_{s-i})=\bigsqcup_{i,j=0}^{\infty}(X_i\times X'_{j})=
\left(\bigsqcup_{i=0}^{\infty}X_i\right)\times\left(\bigsqcup_{j=0}^{\infty}X'_j\right)
$$
such that
$$
\sum_{x\in M}\left(I(\pi_1\circ \psi(x))+I'(\pi_2\circ\psi(x))\right)=k\,.
$$
Here $\pi_1$ and $\pi_2$ are the projections of 
$\left(\bigsqcup\limits_{i=0}^{\infty}X_i\right)\times\left(\bigsqcup\limits_{j=0}^{\infty}X'_j\right)$
to the first and to the second factors respectively. This is the union for $\ell=0, 1, \ldots, k$
of the products of the configuration spaces of maps $\psi_1: M\to \bigsqcup\limits_{i=0}^{\infty}X_i$
and of maps $\psi_2: M\to \bigsqcup\limits_{j=0}^{\infty}X'_j$ with $\sum\limits_{x\in M}I(\psi_1(x))=\ell$
and $\sum\limits_{x\in M}I'(\psi_2(x))=k-\ell$ respectively. This is just a description of the
coefficient of $t^k$ in $\left(A(t)\right)^m\cdot\left(B(t)\right)^m$.

(4) Let $n=[P\stackrel{g}{\rightarrow} Q]$. We have 
$m+n=[M\sqcup P\stackrel{f\sqcup g}{\rightarrow} N\sqcup Q]$.
The coefficient of $t^k$ in $\left(A(t)\right)^{m+n}$ is represented by the configuration space
of maps $\psi$ from $M\sqcup P$ to $\bigsqcup\limits_{i=0}^{\infty}X_i$ such that
$\sum\limits_{x\in M\bigsqcup P}I(\psi(x))=k$. This is the union for $\ell=0, 1,\ldots, k$ of the products
of the configuration spaces of maps $\psi_1:M\to \bigsqcup\limits_{i=0}^{\infty}X_i$
and of maps $\psi_2:P\to \bigsqcup\limits_{i=0}^{\infty}X_i$ with $\sum\limits_{x\in M}I(\psi_1(x))=\ell$ and
$\sum\limits_{x\in P}I(\psi_2(x))=k-\ell$ respectively.
This is just a description of the
coefficient of $t^k$ in $\left(A(t)\right)^m\cdot\left(A(t)\right)^n$.

(5) Let, as above, $n=[P\stackrel{g}{\rightarrow} Q]$. We have 
$mn=[M\times P\stackrel{f\times g}{\rightarrow} N\times Q]$.
The coefficient of $t^s$ in $\left(A(t)\right)^n$ is represented by the configuration space of maps
$\psi:P\to \bigsqcup\limits_{i=0}^{\infty}X_i$ such that $\sum\limits_{x\in P}I(\psi(x))=s$
($s$ is {\em the weight} of the map $\psi$). The coefficient of $t^k$ in $\left(\left(A(t)\right)^n\right)^m$
is represented by the configuration space of maps $\check{\psi}$ from $M$ to the union of the
configuration spaces described above with the sum of weights equal to $k$. Such maps are in one-to-one
correspondence with the maps $\widehat{\psi}:M\times P\to \bigsqcup\limits_{i=0}^{\infty}X_i$
such that $\sum\limits_{x\in M\times P}I(\widehat{\psi}(x))=k$. This is just a description of the
coefficient of $t^k$ in $\left(A(t)\right)^{mn}$.

We have to show that Equation~(\ref{geometric-lambda}) gives:
\begin{equation}\label{Kapran}
 (\boldsymbol{1}+t+t^2+\ldots)^{[M \stackrel{f}{\rightarrow} N]}=\zeta_{[M \stackrel{f}{\rightarrow} N]}(t)\,,
\end{equation}
\begin{equation}\label{conf}
 (\boldsymbol{1}+t)^{[M \stackrel{f}{\rightarrow} N]}=\lambda_{[M \stackrel{f}{\rightarrow} N]}(t)\,.
\end{equation}
The coefficient of $t^k$ in $(1+t+t^2+\ldots)^{[M\stackrel{f}{\rightarrow} N]}$ is represented by the
configuration space of finite subsets of points in $M$ with (positive) multiplicities (since the map
from $M$ to $\bigsqcup\limits_{i=0}^{\infty}pt_i$ is defined by the subset of points $x\in M$ such that
$I(\psi(x))\ne 0$ and by their multiplicities $I(\psi(x))$). This means that this coefficient is equal to
$[S^kM\stackrel{S^kf}{\rightarrow} S^kN]$. This proves (\ref{Kapran}).

The only non-empty summand in the coefficient of $t^k$  in  
 $(\boldsymbol{1}+t)^{[M \stackrel{f}{\rightarrow} N]}$
corresponds to the partition $k_1=k$, $k_i=0$ for $i>1$, and is represented by the map 
$$
(M^k\setminus\Delta)/S_k=B_kM\to N^k/S_k=S^kN.
$$
This proves (\ref{conf}).

\end{proof}

\section{Generating series of Hilbert--Chow morphisms}\label{Hilbert}

Let $\LLL_v\in K_0(\Map)$ be the class of the map $\AA_{\C}^1\to pt$ from the complex affine line to
a one point set: the image of $\LLL$ under the inclusion $K_0(\Map)\subset K_0(\Map)$ defined by
$[X]\mapsto[X\to pt]$. (Do not mix $\LLL_v$ with the the image $\LLL_h$ of the element
$\LLL\in K_0(\Var)$ under the other embedding defined by $[X]\mapsto[X\stackrel{id}{\rightarrow} X]$.)

For a non-singular $d$-dimensional quasi-projective variety $X$, let ${\rm Hilb}^k_X$ be the Hilbert
scheme of zero-dimensional subschemes of length $k$ in $X$. 
One has the Hilbert-Chow morphism $\pi_k:{\rm Hilb}^k_X\to S^kX$ to the $k$th symmetric power of $X$. 
Let ${\rm Hilb}^k_{\C^d,0}$ be the Hilbert
scheme of zero-dimensional subschemes of length $k$ in $\C^d$ supported at the origin.

\begin{theorem}
Let $X$ be a non-singular $d$-dimensional quasi-projective variety. 
 In the Grothendieck ring $K_0(\Map)$ one has
 \begin{equation}\label{Hilb_stat}
 1+\sum_{k=1}^{\infty}[{\rm Hilb}^k_X\stackrel{\pi_k}{\rightarrow} S^kX]\cdot t^k=
\left(1+\sum_{k=1}^{\infty}[{\rm Hilb}^k_{\C^d,0}\to pt]\cdot t^k\right)^{[X \stackrel{id}{\rightarrow} X]}.
 \end{equation}
\end{theorem}

\begin{proof}
The arguments of \cite{Michigan} show that there exists a (finite) Zariski open covering $\{U_i\}$
of $X$ ($X=\bigcup_i U_i$) such that a zero-dimensional subscheme of length $k$ in $U_i$ is determined by
a (finite) collection of points of $U_i$ with a subscheme from
$\bigsqcup\limits_{q=0}^{\infty}{\rm Hilb}^q_{\C^d,0}$ associated to each of them so that the sum of their
lengths $q$ is equal to $k$. Alongside with the geometric description (\ref{geometric-lambda}) of the
power structure over $K_0(\Map)$ this gives
 $$
 1+\sum_{k=1}^{\infty}[{\rm Hilb}^k_{U_i}\stackrel{\pi_k}{\rightarrow} S^kX]\cdot t^k=
\left(1+\sum_{k=1}^{\infty}[{\rm Hilb}^k_{\C^d,0}\to pt]\cdot t^k\right)^{[U_i \stackrel{id}{\rightarrow} U_i]}.
 $$
Applying the inclusion-exclusion formula one gets (\ref{Hilb_stat}).
\end{proof}

For $d=2$ one has
$$
1+\sum_{k=1}^{\infty}[{\rm Hilb}^k_{\C^2,0}\to pt]\cdot t^k=\prod_{i=1}^{\infty}(1-\LLL_v^{i-1}t^i)
\in 1+K_0(\Map)[[t]]\,.
$$
(This is a trivial reformulation of the equation
$$
1+\sum_{k=1}^{\infty}[{\rm Hilb}^k_{\C^2,0}]\cdot t^k=\prod_{i=1}^{\infty}(1-\LLL^{i-1}t^i)
\in 1+K_0(\Var)[[t]]
$$
proved in \cite{Gottsche}). This implies the following statement.

\begin{corollary} For a non-singular quasi-projective surface $X$, one has
 \begin{equation}
 1+\sum_{k=1}^{\infty}[{\rm Hilb}^k_X\stackrel{\pi_k}{\rightarrow} S^kX]\cdot t^k=
\left(\prod_{i=1}^{\infty}(1-\LLL_v^{i-1}t^i)\right)^{[X \stackrel{id}{\rightarrow} X]}\in 1+K_0(\Map)[[t]]\,.
 \end{equation}
\end{corollary}

\section{Versions of the described power structure}\label{sec:versions}

One can see that analogues of the power structure on the Grothendieck ring of maps $K_0(\Map)$
defined by Equation~(\ref{geometric-lambda}) exist in the following settings.

\begin{enumerate}
 \item[{\bf 1.}] {\bf The relative setting.} The Grothendieck group $K_0(\Map/\varphi)$
 of maps over a fixed map $\varphi:S_1\to S_2$ is defined as the Grothendieck group generated by
 the classes of commutative diagrams of the form
 $$ 
\begin{CD}
X@>f>> Y\\
@VVV                @VVV\\
S_1 @>{\rm \varphi}>> S_2
\end{CD}
 $$ 
 with the natural analogues of the relations 1)--3). The multiplication in $K_0(\Map/\varphi)$ is defined by
 the fibre products over $S_1$ and $S_2$. An analogue of Equations~(\ref{geometric-lambda}) defines an 
 (effective) power structure over $K_0(\Map/\varphi)$. In this analogue all the products
 (including $M^{\sum k_i}$ and $N^{\sum k_i}$ considered as products of $\sum k_i$ copies of $M$ and of $N$
 respectively) are fibre products over $S_1$ or $S_2$.

 \item[{\bf 2.}] {\bf The equivariant setting.} For a finite group $G$, the Grothendieck ring $K_0^G(\Map)$
 of G-equivariant maps is defined as the Grothendieck ring generated by the classes
 $[X\stackrel{f}{\rightarrow}Y]$, where $X$ and $Y$ are complex quasi-projective $G$-varieties and $f$
 is a $G$-equivariant map. All the maps in  (\ref{geometric-lambda}) are $G$-equivariant.

\item[{\bf 3.}] {\bf The relative setting over an Abelian monoid.} Let $\mathfrak{M}$ be an Abelian monoid
 with zero. As in the relative setting above, the Grothendieck group $K_0(\Map/\mathfrak{M})$
 of maps over the monoid $\mathfrak{M}$ is defined as the Grothendieck group generated by the classes
 of commutative diagrams of the form
 $$ 
\begin{CD}
X@>f>> Y\\
@VV{P_X}V                @VV{p_Y}V\\
\mathfrak{M} @>{\rm id}>> \mathfrak{M}
\end{CD}
 $$ 
The difference is in the definitions of the multiplication in $K_0(\Map/\mathfrak{M})$
and of the maps to $\mathfrak{M}$ of the summands in Equations~(\ref{geometric-lambda}).
The multiplication is defined via the map $\mathfrak{M}\times\mathfrak{M}\to\mathfrak{M}$
applied to the usual Cartesian product (with the target $\mathfrak{M}\times\mathfrak{M}$).
To define the map to $\mathfrak{M}$ on the summands of (\ref{geometric-lambda}), it is
useful to use consider them as configuration spaces of particles on $M$ with some charges
and some weights. The weights of a particle $s\in M$ of charge $n$ (and thus of the internal state
$\phi$ from the variety $X_n$) is defined as $np_M(s)+p_{X_n}(\phi)$, where $p_M$ and $p_{X_n}$ are
the maps from the corresponding varieties to $\mathfrak{M}$. The weight of a collection of particles
is defined as the sum of the weights of the individual particles. Cf. \cite{M_Sh}. 

\item[{\bf 4.}] {\bf The relative setting over a morphism of Abelian monoids.} One can see that
the definition of the Grothendieck ring and of the power structure over it from {\bf 3} can be extended 
to maps over a fixed morphism of Abelian monoids $\varphi:\mathfrak{M}_1\to\mathfrak{M}_2$.
\end{enumerate}

\end{document}